\documentclass[10.5pt, oneside]{article}   	
\usepackage{geometry}                		
\usepackage{graphicx}				
\usepackage{amssymb, amsmath, amscd}
\usepackage{mathtools}
\mathtoolsset{showonlyrefs=true}		

\usepackage{amsthm}
\theoremstyle{plain}
\newtheorem{thm}{Theorem}[section]
\newtheorem{prop}{Proposition}[section]
\newtheorem{cor}{Corollary}[section]
\newtheorem{lem}{Lemma}[section]

\makeatletter						
\@addtoreset{equation}{section}

\makeatother
\def\RR{\mathbb{R}}
\def\SS{\mathbb{S}}
\def\inn<#1>{\langle #1 \rangle} 		
\def\dist(#1){\mathrm{dist}^{\Sigma}(#1)} 	
\def\tr{\mathrm{tr}}					
\def\det{\mathrm{det}}				
\def\nablaS{\nabla^{\Sigma}} 			
\def\lapS{\Delta^{\Sigma}}				
\def\divS{\mathrm{div}^\Sigma}			
\def\intS{\int_{\Sigma}}				
\def\RmS(#1){\mathrm{Rm}^{\Sigma}(#1)}	
\def\RicS(#1){\mathrm{Ric}^{\Sigma}(#1)}	
\def\anF{\mathcal{F}_\gamma}			
\def\aniso{\mathrm{tr}(A_\gamma S^2)}	
\def\Ind{\mathrm{Ind}}				

\title{On the Gauss Map of Anisotropic Minimal Surfaces and applications to the Morse Index estimates}
\author{Toshimi Inoue}
\date{}							

\begin{document}
\maketitle
\begin{abstract}
In the paper, we study the Gauss map of a completely immersed anisotropic minimal surface with respect to convex parametric integrand in $\RR^3$. By a local analysis, we prove the discreteness of the critical set of the Gauss map of an anisotropic minimal surface. In particular, we may consider the Gauss map as a branched covering map from an anisotropic minimal surface to the unit sphere. As a consequence, we may obtain an upper and a lower estimate for the Morse index of an anisotropic minimal surface by applying some conformal geometric technics to the Gauss map. 
\end{abstract}
\section{Introduction}
Let $\gamma : \SS^2 \to \RR_{>0}$ be a smooth positive function on the unit sphere. For an isometric immersion $\Sigma \to \RR^3$ from an oriented surface, the {\it anisotropic surface energy} $\anF(\Sigma)$ is defined by 
\begin{equation}\label{eq:anisotropicenergy}
	\anF(\Sigma) = \intS \gamma(\nu)d\Sigma, 
\end{equation}
where $\nu$ is the unit normal vector field along $\Sigma$ and $d\Sigma$ is the area element of $\Sigma$. 

A surface $\Sigma$ is called a {\it $\gamma$-anisotropic minimal surface} if it satisfies $\left.\frac{d}{dt}\right|_{t=0}\anF(\Sigma_t) = 0$ for all compactly supported variations. We note that if we let $\gamma = 1$ in \eqref{eq:anisotropicenergy}, then $\anF$ coincides with the area functional. In this sense, we may consider anisotropic minimal surfaces as a generalization of minimal surfaces. 

The stability and the instability of solutions to the variational problem of given functional is an important problem in differential geometry. For minimal surfaces, the Gauss map is a fundamental tool in the study of the stability. Barbosa and Do Carmo \cite{BD} gave a sufficient condition for stability of subdomains of minimal surfaces by studying the image of Gauss map. Tysk \cite{T} studied the spectrum of the Laplacian of covering manifolds and applied it to the Gauss map of a minimal surface to get an upper bound of the Morse index in terms of the total curvature. Nayatani \cite{N} gave a lower bound for the Morse index by considering the pre-image  of a great circle in the unit sphere via the Gauss map. 

In this article, we study the behavior of the Gauss map of anisotropic minimal surfaces with convex integrand $\gamma$ and apply it to obtain an estimate of the Morse index . A positive function $\gamma$ on the unit sphere is called {\it convex} if the second derivative $A_\gamma = (\nabla^{\SS^2})^2\gamma+\gamma I$ of $\gamma$ is positively definite. When $\gamma$ is convex, the Euler--Lagrange equation of the variation of $\anF$ becomes an elliptic equation. Moreover, the Jacobi operator, which is defined through the second variation of $\anF$, becomes a second order linear elliptic operator and its Morse index can be well-defined. For the behavior of the Gauss map, we obtain the following result. 
\begin{thm}\label{thm:main1}
Let $\Sigma \subset \RR^3$ be a $\gamma$-anisotropic minimal surface. Assume $\gamma$ is convex. Then the critical set of the Gauss map of $\Sigma$ has no accumulation points.
\end{thm}
This result is a generalization of the well known fact for minimal surfaces. The discreteness of the critical set of the Gauss map allows us to consider the Gauss map as a branched covering map onto the unit sphere. In the anisotropic case, we define the {\it anisotropic Gauss map} $\nu_\gamma$ by the composition of the Gauss map $\nu: \Sigma \to \SS^2$ and $\xi: \SS^2 \to \RR^3$ defined by $\xi(\nu) = \nabla^{\SS^2}\gamma(\nu) + \gamma(\nu)\nu$ for $\nu \in \SS^2$. Applying an argument in \cite{T} to the anisotropic Gauss map, we obtain the following index upper bound. 
\begin{thm}\label{thm:main2}
Let $\Sigma \subset \RR^3$ be a $\gamma$-anisotropic minimal surface with convex integrand $\gamma$. Then there exists a positive constant $C(\gamma)$ which depends only on $\gamma$ such that 
\begin{equation}
	\Ind(\Sigma) \leq C(\gamma)\mathrm{deg}(\nu_\gamma), 
\end{equation}
where $\mathrm{deg}(\nu_\gamma)$ is the mapping degree of the anisotropic Gauss map $\nu_\gamma$ on $\Sigma$. 
\end{thm}

Another application of Theorem\ref{thm:main1} is the following generalization of Nayatani's Morse index lower bound in \cite{N}. 
\begin{thm}\label{thm:main3}
Let $\Sigma \subset \RR^3$ be a $\gamma$-anisotropic minimal surface with convex integrand $\gamma$. Let $S$ be a great circle in the unit sphere $\SS^2$. Then, the Morse index $\Ind(\Sigma)$ of $\Sigma$ satisfies
\begin{equation}
	\Ind(\Sigma) \geq b(\nu, \nu^{-1}(S)) + 1 -2g, 
\end{equation}
where $b(\nu, \nu^{-1}(S))$ is the branching order of the Gauss map $\nu$ with respect to $\nu^{-1}(S)$ and $g$ is the genus of $\Sigma$. 
\end{thm}

This article is organized as follows. In section 2, we prepare some basic notions and properties of the geometry of anisotropic minimal surfaces. We also give a proof of Theorem \ref{thm:main1} in this section. In section 3, we study the anisotropic Gauss map of anisotropic minimal surfaces and give a proof of Theorem \ref{thm:main2}. Section 4 contains a generalization of the index lower bound by Nayatani to the case of anisotropic minimal surfaces. We will apply it to studying the instability for anisotropic minimal surfaces with low genus. 


\section{Geometry of anisotropic minimal surfaces}
In this section, we prepare some basic properties of $\gamma$-anisotropic minimal surfaces. We refer to \cite{K} and \cite{KP1}. Let $\Sigma \to \RR^3$ be an immersion of an oriented surface. Let $\gamma: \SS^2 \to \RR_{>0}$ be a smooth positive function on the unit sphere. For any variation $\Sigma_t$ of $\Sigma$, the first variation of the anisotropic surface energy $\anF$ defined by \eqref{eq:anisotropicenergy} is given by
\begin{equation}\label{eq:1stvariation}
	\left.\frac{d}{dt}\right|_{t=0}\anF(\Sigma_t) = -\intS H_\gamma u d\Sigma, 
\end{equation}
where $u$ is the normal component of given variation and $H_\gamma$ is the {\it anisotropic mean curvature}. As a consequence of \eqref{eq:1stvariation}, we see that $\Sigma$ is a critical point of the functional $\anF$ if and only if its anisotropic mean curvature vanishes identically. 

The integrand $\gamma$ is called {\it convex} or {\it elliptic} if the hessian  $A_\gamma = (\nabla^{\SS^2})^2\gamma+\gamma I$ of $\gamma$ (extended 1-homogeneously) satisfies $A_\gamma > 0$. Geometrically speaking, the convexity of $\gamma$ implies that $\gamma$ is the support function of the convex body 
\begin{equation}\label{eq:Wulffbody}
	\bigcap_{\nu \in \SS^2} \{x \in \RR^3 | \inn<x, \nu> \leq \gamma(\nu)\}. 
\end{equation}
The boundary $W_{\gamma}$ of the convex set \eqref{eq:Wulffbody} is called the {\it Wulff shape} of $\gamma$. We note that if $\gamma$ is convex, the Wulff shape $W_{\gamma}$ has a parametrization
\begin{equation}\label{eq:Wulffshape}
	W_\gamma = \xi(\SS^2), 
\end{equation}
where $\xi$ is the map, which is called the {\it Cahn--Hoffman map for $\gamma$} (see \cite{K} for details), defined by $\xi(\nu) = \nabla^{\SS^2}\gamma(\nu) + \gamma(\nu)\nu$ from $\SS^2$ to $W_\gamma$. The following fact is a fundamental property for the Cahn--Hoffman map (see \cite{Sch}). 

\begin{lem}\label{lem:CHnormal}
Let $\gamma : \SS^2 \to \RR_{>0}$ be a smooth positive function on $\SS^2$. Then the Cahn--Hoffman map $\xi$ for $\gamma$ satisfies $\inn<d\xi_{\nu}(X), \nu> = 0$ for any $\nu \in \SS^2$ and $X \in T_{\nu}\SS^2$. 
\end{lem}
%
By Lemma \ref{lem:CHnormal}, the tangent planes $T_\nu{\SS^2}$ at $\nu \in \SS^2$ and $T_{\xi(\nu)}{{W_\gamma}}$ at $\xi(\nu) \in W_\gamma$ are parallel. In particular, we may consider the differential $d\xi = A_\gamma$ at $\nu \in \SS^2$ of the Cahn--Hoffman map $\xi$ as an endomorphism on $T_{\nu}\SS^2$. 

For a surface $\Sigma$, the map $\nu_{\gamma} = \xi \circ \nu$ from $\Sigma$ to $W_\gamma$ is called the {\it anisotropic Gauss map}. Since the planes $T_{p}\Sigma$ at $p \in \Sigma$ and $T_{\nu_{\gamma}(p)}W_\gamma$ at $\nu_{\gamma}(p) \in W_\gamma$ are parallel, we may consider the differential $d\nu_{\gamma} = -A_{\gamma}S$ as an endomorphism on $T_{p}\Sigma$. Here, $S$ denotes the usual shape operator of $\Sigma$. 
Hence, the anisotropic mean curvature can be written as 
\begin{equation}\label{eq:Ameancurve}
H_\gamma = \tr(-d\nu_\gamma) = \tr(A_\gamma S). 
\end{equation}

Let us assume that $\Sigma$ is a $\gamma$-anisotropic minimal surface with convex integrand $\gamma$. We give a local expression of \eqref{eq:Ameancurve}. Given $p \in \Sigma$, choose a frame $\{e_i\}_i$ diagonalizing the second fundamental form $A$ and let $\{\kappa_i\}_i$ be the principal curvatures. We set $a_i = \inn<A_\gamma e_i, e_i>$. Note that $a_1$ and $a_2$ are positive by the convexity of $\gamma$. Under these notations, we have 
\begin{equation}\label{eq:locAmean}
	H_\gamma = a_1\kappa_1+a_2\kappa_2. 
\end{equation}

Using \eqref{eq:locAmean}, the Gaussian curvature $K^\Sigma$ of $\Sigma$ is given by
\begin{equation}\label{eq:Gausscurve}
	K^\Sigma = \kappa_1\kappa_2 = -\frac{a_1}{a_2}{\kappa_1}^2 \leq 0. 
\end{equation}
Here, we used the convexity of $\gamma$. Therefore, a $\gamma$-anisotropic minimal surface has nonpositive Gaussian curvatute. 

Before showing Theorem \ref{thm:main1}, we consider a graphical surface. Let $u = u(x,y)$ be a smooth function on a domain $\Omega \subset \RR^2$. Since the unit normal $\nu$ of $ \mathrm{Graph}_u$ is given by
\begin{equation}
	\nu = \frac{(-u_x, -u_y, 1)}{\sqrt{1 + |\nabla u|^2}}, 
\end{equation}
the anisotropic energy $\anF$ of $\mathrm{Graph}_u$ becomes
\begin{equation}\label{eq:anisograph}
	\anF(\mathrm{Graph}_u) = \int_{\Omega} \gamma(\nu) \sqrt{1+|\nabla u|^2}\,dxdy 
	= \int_{\Omega} \bar{\gamma}(-u_x, -u_y, 1)\,dxdy. 
\end{equation}
Taking a variation of \eqref{eq:anisograph}, we obtain the $\gamma$-anisotropic minimal surface equation: 
\begin{equation}\label{eq:anisographeq}
	\sum_{i, j = 1}^{2} \bar{\gamma}_{ij}(-\nabla{u}, 1) u_{ij} = 0. 
\end{equation}

We next show a curvature estimate for graphical surfaces. Using $(x, y)$ as local coordinates on $\mathrm{Graph}_u$, the induced metric $g$ and its inverse $g^{-1}$ can be written as 
\begin{equation}
	 g = 
	\begin{pmatrix}
		1 + u_{x}^2 & u_x u_y\\
		u_x u_y & 1 + u_{y}^2
	\end{pmatrix}
	,\quad g^{-1} =\frac{1}{1+|\nabla u|^2}
	 \begin{pmatrix}
		1 + u_{y}^2 & -u_x u_y\\
		-u_x u_y & 1 + u_{x}^2
	\end{pmatrix}
	.
\end{equation}
Since the eigenvalues of $g$ are $1$ and $1 + |\nabla u|^2$, $g^{-1}$ satisfies the following inequality as a symmetric matrix
\begin{equation} \label{eq:inverseesti}
	\frac{1}{1+|\nabla u|^2} I \leq g^{-1} \leq I. 
\end{equation}

Similarly, the second fundamental form $A$ in the coordinate $(x, y)$ is expressed as 
\begin{equation}
	A = \frac{1}{(1 + |\nabla u|^2)^{\frac{1}{2}}}
	\begin{pmatrix}
		u_{xx} & u_{xy}\\
		u_{xy} & u_{yy}
	\end{pmatrix}
	.
\end{equation}
The norm squared of $A$ is given by $|A|^2 = A_{ij}A_{kl}\,g^{ik}g^{jl}$. Hence the matrix inequality \eqref{eq:inverseesti} yields the following curvature estimate:
\begin{equation}\label{eq:graphAesti}
	\frac{|\mathrm{Hess}_u|^2}{(1 + |\nabla u|^2)^3} \leq |A|^2 \leq \frac{|\mathrm{Hess}_u|^2}{1 + |\nabla u|^2}.
\end{equation}
Finally, we recall the following fact from \cite{CM}. 
\begin{lem}\label{lem:smallcurve}
Let $\Sigma \subset \RR^3$ be an immersed surface with
\begin{equation}
	16s^2\sup_{\Sigma}|A| \leq 1
\end{equation}
for some $s>0$. If $x_0 \in \Sigma$ and $\dist(x_0, \partial \Sigma) \geq 2s$, then $B^\Sigma_{2s}(x_0)$ can be written as a graph of a functtion $u$ over a domain of $T_{x_0}\Sigma$ with 
	$|\nabla u| \leq 1$ and $\sqrt{2}|\mathrm{Hess}_u| \leq 1$.
\end{lem}

\begin{proof}[Proof of Therorem\ref{thm:main1}]
Let $p \in \Sigma$ be a point in $\Sigma$ such that $K^\Sigma(p) = 0$. Let $\kappa_i$ be the $i$-th principal curvature of $\Sigma$ at $p$. It follows from \eqref{eq:Gausscurve} that
\begin{equation}
	|A|^2 = \kappa_1^2 + \kappa_2^2 = \kappa_1^2 + \frac{a_1^2}{a_2^2}\kappa_1^2 = \left(\frac{a_2}{a_1}\ + \frac{a_1}{a_2} \right)\frac{a_1}{a_2}\kappa_1^2 = \left(\frac{a_2}{a_1}\ + \frac{a_1}{a_2} \right) (-K^\Sigma), 
\end{equation} 
so $|A|^2(p) = 0$. 

Up to translation and rotation, we may assume that $p = 0$, $T_p\Sigma = \{x_3 = 0\}$, and $\nu(p) = (0, 0, 1)$.

Since $|A|^2(0) = 0$, there exists a $\delta > 0$ such that $|A|^2 < 1$ on $B^\Sigma_{\delta}(0)$. Rescaling the canonical metric on $\RR^3$ by factor $\lambda > 0$, we have
\begin{equation}
	|\tilde{A}|^2 < \frac{1}{\lambda^2} \quad \text{on $B^\Sigma_{\lambda \delta}(0)$}, 
\end{equation}
where $\tilde{A}$ denotes the second fundamental form of $\Sigma$ given by the rescaled metric. Choosing $\lambda > 0$ large, we may assume that 
\begin{equation}
	|A|^2 \leq \frac{1}{16} \quad \text{on $B^\Sigma_4(0)$}.
\end{equation}
By Lemma \ref{lem:smallcurve} the geodesic disk $B^\Sigma_2(0)$ can be written as a graph of a function $u$ over a domain in $\RR^2 = T_p\Sigma$ with $|\nabla u| \leq 1$ and $|\mathrm{Hess}_u| \leq \frac{1}{\sqrt{2}}$. 

Since $\Sigma$ is $\gamma$-minimal, $u$ satisfies a quasilinear elliptic equation \eqref{eq:anisographeq} with bounded coefficients. 
Up to linear transform of local coordinate, we may assume that $\bar{\gamma}_{ij}(-\nabla{u}(0), 1) = \delta_{ij}$. In this coordinate, $u$ satisfies $\Delta u = 0$ at $p$. 

By the theorem of Bers \cite{B}, there exists a homogeneous polynomial $P$ of degree $n$ and a constant $\varepsilon \in (0, 1)$ such that 
\begin{align}
	u(x) &= P(x) + \mathcal{O}(|x|^{n +\varepsilon}), \\
	u_i(x) &= P_i (x) + \mathcal{O}(|x|^{n -1 +\varepsilon}), \quad i = 1, 2, \\
	u_{ij}(x) &= P_{ij}(x) + \mathcal{O}(|x|^{n -2 +\varepsilon}), \quad 1 \leq i, j \leq 2, \\
\end{align}
and
\begin{equation}\label{eq:harmP}
	\Delta P = 0
\end{equation}
hold around $0$. Since $u(0) = 0$ and $\nabla{u}(0) = 0$, it follows that $n \geq 2$. 

Moreover, by \eqref{eq:graphAesti}, we have
\begin{equation}
	|\mathrm{Hess}_u|^2 \leq (1 + |\nabla{u}|^2)^3 |A|^2 \leq 8|A|^2, 
\end{equation}
which implies $u_{ij}(0) = 0$ and hence $n \geq 3$. 

We now introduce the complex coordinate $z$ on $\RR^2 = T_p\Sigma$ by $z = x_1 + \sqrt{-1} x_2$. The harmonicity \eqref{eq:harmP} of $P$ yields $P_{z\bar{z}} = 0$. Hence, the function $P_z$ is holomorphic. 

Since $P$ is a polynomial of degree $n \geq 3$ and $P_z$ is holomorphic, there exists a holomorphic function $G$ around $0$ such that
\begin{equation}
	P_{zz}(z) = z^{n-2}\,G(z)\quad \text{and} \quad G(0) \neq 0.
\end{equation}

Assume that there exists a sequence $\{z_k\}_k$ such that $z_k \to 0$ and $|A|(z_k) = 0$. Then $u_{ij}(z_k) = 0$ for each $k$, and hence $u_{zz}(z_k) = 0$. 

On the other hand, we have 
\begin{align}
	u_{zz}
		&= P_{zz} + \mathcal{O}(|z|^{n -2 +\varepsilon})\\
		&= z^{n - 2}\,(G(z) + \mathcal{O}(|z|^{\varepsilon})). 
\end{align}
For each $z_k$, we get 
\begin{equation}
	0 = u_{zz}(z_k) = z_k^{n - 2}\, (G(z_k) + \mathcal{O}(|z_k|^{\varepsilon})). 
\end{equation}
Since $n\geq3$, this implies $G(z_k) = \mathcal{O}(|z_k|^{\varepsilon})$ which contradicts the fact that $G(0) \neq 0$. 

Therefore, the set $\{p \in \Sigma | K^{\Sigma}(p) = 0\}$ has no accumulation points, which proves the theorem. 
\end{proof}

\section{Index upper bound for anisotropic minimal surfaces}
Let $\Sigma$ be a complete, oriented $\gamma$-anisotropic minimal surface. Recall that the second variation formula for the anisotropic surface energy $\anF$ is given by
\begin{equation}
	\left. \frac{d^2}{dt^2} \right|_{t = 0} \anF = -\int_\Sigma uL u d\Sigma, 
\end{equation}
where $L$ is the \textit{Jacobi operator} given by $L = \divS(A_\gamma \nablaS) + \aniso$. For relatively compact domain $\Omega \subset \Sigma$, we define $\Ind(\Omega)$ as the number of negative eigenvalues (counted with multiplicity) of the Dirichlet eigenvalue problem
\begin{equation}
	L u + \lambda u = 0\,\, \text{in $\Omega$},\quad u = 0\,\, \text{on $\partial{\Omega}$}.
\end{equation}
Note that the number $\Ind(\Omega)$ is always finite by the theory of elliptic operators. The {\it Morse index} $\Ind(\Sigma)$ of $\Sigma$ is defined as the supremum of the numbers $\Ind(\Omega)$ over all relatively compact domains in $\Sigma$. The associated quadratic form to $L$ is given by
\begin{equation}\label{eq:indexform}
	Q(u) = \intS (\inn<A_\gamma \nablaS u, \nablaS u> - \aniso u^2)d\Sigma. 
\end{equation}

We set
\begin{equation}
	\lambda_\gamma = \min_{\nu \in \SS^2}\min_{v \in T_{\nu}\SS^2, |v| = 1}\inn<A_\gamma v, v>, 
	\quad \Lambda_\gamma = \max_{\nu \in \SS^2}\max_{v \in T_{\nu}\SS^2, |v| = 1}\inn<A_\gamma v, v>. 
\end{equation}
By the convexity of $\gamma$, we have 
\begin{equation}\label{eq:upperandlower}
	0 < \lambda_\gamma I  \leq A_\gamma \leq \Lambda_\gamma I, 
\end{equation}
where $I$ denotes the identity matrix.

To estimate the quadratic form, we define the {\it anisotropic Gaussian curvature} $K_\gamma$ by 
\begin{equation}\label{eq:anisoGausscurve}
	K_\gamma = \det(-d\nu_\gamma) = \det(A_\gamma S) = \det(A_\gamma) K^\Sigma. 
\end{equation}

\begin{lem}\label{lem:anisoidentity}
Let $\Sigma$ be a $\gamma$-anisotropic minimal surface with convex integrand $\gamma$. We have
\begin{equation}
	-\frac{2}{\Lambda_\gamma}K_\gamma \leq \aniso \leq -\frac{2}{\lambda_\gamma}K_\gamma. 
\end{equation}
\end{lem}
\begin{proof}
Let $\kappa_i$ be the principal curvature. Using the identity $a_1\kappa_1+a_2\kappa_2 = 0$, we have 
\begin{equation}
	\aniso = a_1{\kappa_1}^2+a_2{\kappa_2}^2 = -a_2\kappa_1\kappa_2-a_1\kappa_1\kappa_2 = -(a_1+a_2)K^\Sigma = -\tr(A_\gamma)K^\Sigma = -\frac{\tr(A_\gamma)}{\det{A_\gamma}}K_\gamma. 
\end{equation}
If we let $\mu_1$ and $\mu_2$ be the eigenvalues of $A_\gamma$, we have
\begin{equation}
	\frac{\tr(A_\gamma)}{\det{A_\gamma}} = \frac{\mu_1+\mu_2}{\mu_1 \mu_2} = \frac{1}{\mu_1} + \frac{1}{\mu_2}. 
\end{equation}
Moreover, since \eqref{eq:upperandlower} gives $\lambda_\gamma \leq \mu_1, \mu_2 \leq \Lambda_\gamma$, we obtain
\begin{equation}
	\frac{2}{\Lambda_\gamma} \leq \frac{\tr(A_\gamma)}{\det{A_\gamma}} \leq \frac{2}{\lambda_\gamma}. 
\end{equation} 
\end{proof}
By Lemma \ref{lem:anisoidentity}, the quadratic form $Q$ can be estimated below by
\begin{align}
	Q(u) 
		&= \intS (\inn<A_\gamma \nablaS u, \nablaS u> - \aniso u^2)d\Sigma \\
		&\geq  \intS \left(\lambda_\gamma |\nablaS u|^2 + \frac{2}{\lambda_\gamma}K_\gamma u^2 \right)d\Sigma \\
		&= \lambda_\gamma Q_{\gamma}(u), \label{eq:Qlowerbound}
\end{align}
where we defined $Q_{\gamma}(u) = \intS \left(|\nablaS u|^2 + \frac{2}{{\lambda_\gamma}^2}K_\gamma u^2 \right)\, d\Sigma$, which is the quadratic form corresponding to the elliptic operator $L_\gamma = \lapS - \frac{2}{{\lambda_\gamma}^2}K_\gamma$. By \eqref{eq:Qlowerbound}, if $u$ satisfies $Q(u)<0$, it also follows that $Q_\gamma(u)<0$. This implies that the Morse index can be bounded from above by the number of negative eigenvalues $\mathrm{Neg}(L_\gamma)$ of $L_\gamma$. 

We now assume that $\Sigma$ has finite total curvature. Since the Gaussian curvature of $\gamma$-anisotropic minimal surface is nonpositive by \eqref{eq:Gausscurve}, the theorem of Osserman \cite[Theorem 9.1]{Oss} asserts that there exists a compact surface $\tilde{\Sigma}$ such that $\Sigma$ is conformally diffeomorphic to $\tilde{\Sigma} \setminus \{p_1, \cdots p_k\}$. By the theorem of White \cite{W}, the Gauss map $\nu$ can be extended continuously to $\tilde{\Sigma}$. Thus, it follows from Theorem\ref{thm:main1} that the set $C = \{p \in \tilde{\Sigma} | K^\Sigma(p) = 0\}$ is a finite set. Since $C$ is the critical set of $\nu$, $\nu$ is locally diffeomorphic outside of $C$. Letting $B = \nu(C) \subset \SS^2$, we obtain the following: 

\begin{lem}\label{lem:covering}
The restricted Gauss map $\nu: \tilde{\Sigma} \setminus \nu^{-1}(B) \to \S \setminus B$ is a covering map. Moreover, $\nu: \tilde{\Sigma} \to \SS^2$ is surjective. 
\end{lem}

Since $\gamma$ is convex, \eqref{eq:anisoGausscurve} asserts that the critical set of $\nu_\gamma$ coincides with $C$. Hence, the anisotropic Gauss map $\nu_\gamma : \tilde{\Sigma} \setminus \nu^{-1}(B) \to W_\gamma \setminus \xi(B)$ is also a covering map onto the Wulff shape $W_\gamma$. 

Now, let $h$ be the metric on $W_\gamma$ induced from the immersion $W_\gamma \to \RR^3$. In the case of minimal surfaces, the pullback metric from $\SS^2$ via the Gauss map is conformal to the original one. Since the index form for the Jacobi operator on a minimal surface is invariant under conformal changes, we may relate the Morse index with the eigenvalue problem of the Laplacian on $\SS^2$. Motivated by these facts, we consider the energy density of given function with respect to the pullback metric $\nu_\gamma^*h$ on $\Sigma$. 

\begin{lem}[cf. Lemma 3.10 in \cite{BD}]\label{lem:energyestimate}
Let $\Sigma$ be a $\gamma$-anisotropic minimal surface with convex integrand $\gamma$. Assume that $p$ is a point at which $\nu_\gamma$ is regular. Then there exists a constant $c(\gamma)>0$ which depends only on $\gamma$ such that it follows for any smooth function $u$ on $\Sigma$ that 
\begin{equation}
	-c(\gamma)^{-1}\frac{1}{K_{\gamma}}|du|^2(p) \leq |du|_{\nu_\gamma^*h}^2(p) \leq -c(\gamma)\frac{1}{K_{\gamma}}|du|^2(p). 
\end{equation}
\end{lem}
\begin{proof}
Let $\{e_i\}_i$ be an orthonormal frame around $p$. Since the tangent planes $T_p\Sigma$ and $T_{\nu_\gamma(p)}W_\gamma$ are parallel, we may consider $\{e_i(p)\}_i$ as an orthonormal basis of $T_{\nu_\gamma(p)}W_\gamma$. In this basis, the metric $\nu_\gamma^*h$ is expressed as 
\begin{equation}
	(\nu_\gamma^*h)_{ij} = \inn<d\nu_\gamma(e_i), d\nu_\gamma(e_j)> = \inn<A_\gamma Se_i, A_\gamma Se_j> = (SA_\gamma^2S)_{ij}. 
\end{equation}
To obtain the last equality, we used the symmetry of $A_\gamma$ and $S$. 

Note that the matrix $SA_\gamma^2S$ is symmetric. Choose $\{e_i(p)\}_i$ so that it diagonalizes $SA_\gamma^2S$ and let $\lambda_1$ and $\lambda_2$ be the corresponding eigenvalues. Since $\nu_\gamma$ is regular at $p$, the matrix $SA_\gamma^2S$ is invertible. Hence, we have
\begin{equation}
	|du|_{\nu_\gamma^*h}^2(p) = \inn<(SA_\gamma^2 S)^{-1}du, du> = {\lambda_1}^{-1}u_1^2+{\lambda_2}^{-1}u_2^2. 
\end{equation}
This equality yields
\begin{equation}\label{eq:energy1}
	\frac{1}{\tr(SA_\gamma^2 S)}|du|^2=\frac{1}{\lambda_1+\lambda_2}|du|^2 \leq |du|_{\nu_\gamma^*h}^2(p) \leq \frac{\lambda_1+\lambda_2}{\lambda_1\lambda_2}|du|^2 = \frac{\tr(SA_\gamma^2 S)}{\det(SA_\gamma^2 S)}|du|^2. 
\end{equation}

Since $A_\gamma^2$ satisfies $\lambda_\gamma^2 I \leq A_\gamma^2 \leq \Lambda_\gamma^2 I$, we have
\begin{equation}
	\inn<SA_\gamma^2Sv, v> = \inn<A_\gamma^2Sv, Sv> \leq \Lambda_\gamma^2\inn<S^2v, v>, 
\end{equation}
for any vector $v \in T_p\Sigma$. Hence, 
\begin{align}
	\tr(SA_\gamma^2S) 
		&\leq \Lambda_\gamma^2\tr(S^2) = \Lambda_\gamma |A|^2 = -\Lambda_\gamma^2\left(\frac{a_2}{a_1}+\frac{a_1}{a_2}\right)K^\Sigma\\
		&\leq -\Lambda_\gamma^2\left(\frac{\Lambda_\gamma}{\lambda_\gamma}+\frac{\lambda_\gamma}{\Lambda_\gamma}\right)K^\Sigma\\
		&\leq -\left(\frac{\Lambda_\gamma}{\lambda_\gamma}\right)^2\left(\frac{\Lambda_\gamma}{\lambda_\gamma}+\frac{\lambda_\gamma}{\Lambda_\gamma}\right)K_\gamma\\
		&= -c(\gamma)K_\gamma, 
\end{align}
by \eqref{eq:anisoGausscurve}. Combining this with \eqref{eq:energy1}, we obtain
\begin{equation}
	-c(\gamma)^{-1}\frac{1}{K_\gamma}|du|^2 \leq |du|_{\nu_\gamma^*h}^2 \leq -c(\gamma)\frac{1}{K_\gamma}|du|^2. 
\end{equation}
\end{proof}

Before the proof of the index upper bound, we prepare the following notation. For any second order linear differential operator $A$ and any subdomain $\Omega$ of $\Sigma$, let $N_\lambda(A, \Omega)$ be the number of the eigenvalues of $A$ less than $\lambda$ on $\Omega$. 
\begin{proof}[Proof of Theorem \ref{thm:main2}]
If the degree of $\nu_\gamma$ is infinity, there are nothing to prove. Assume that the degree of $\nu_\gamma$ is finite. 

Since
\[
	-\lambda_{\gamma}^2\intS K^\Sigma d\Sigma \leq \mathrm{deg}(\nu_\gamma) \leq -\Lambda_{\gamma}^2\intS K^\Sigma d\Sigma, 
\]
by \eqref{eq:upperandlower}, the finiteness of $\mathrm{deg}(\nu_\gamma)$ implies the finiteness of total curvature. Thus, it follows from Lemma $\ref{lem:covering}$ that $\nu_\gamma$ is a branched covering map from $\Sigma$ to $W_\gamma$. 

For each $\varepsilon>0$, let $\Omega_\varepsilon$ be the subset of $\Sigma$ subtracting the $\varepsilon$-neighborhood of $\nu^{-1}(B)$.  
By Lemma \ref{lem:energyestimate}, we have
\begin{align}
	Q_\gamma(u) 
	&=  \intS \left(|du|^2 + \frac{2}{{\lambda_\gamma}^2}K_\gamma u^2 \right)d\Sigma\\
	&\geq  \intS \left(-c(\gamma)^{-1}K_\gamma |du|_{\nu_\gamma^*h}^2 + \frac{2}{{\lambda_\gamma}^2}K_\gamma u^2 \right)d\Sigma\\
	&= c(\gamma)^{-1}\intS \left(|du|_{\nu_\gamma^*h}^2 - \frac{2c(\gamma)}{{\lambda_\gamma}^2}u^2 \right)(-K_\gamma)d\Sigma\\
	&=  c(\gamma)^{-1}\intS \left(|du|_{\nu_\gamma^*h}^2 - \frac{2c(\gamma)}{{\lambda_\gamma}^2}u^2 \right)d(\nu_{\gamma}^{*}h)
\end{align}
for any function $u$ with compact support in $\Omega_\varepsilon$. This implies that 
\[
	N_0(L_\gamma, \Omega_\varepsilon) \leq N_{c'(\gamma)}(\Delta^{\nu_{\gamma}^{*}h}, \Omega_\varepsilon), 
\]
where $c'(\gamma) = 2c(\gamma)\lambda_\gamma^{-2}$ and $\Delta^{\nu_{\gamma}^{*}h}$ is the Laplacian with respect to $\nu_{\gamma}^{*}h$. Letting $\varepsilon \to 0$, we obtain
\[
	\Ind(\Sigma) \leq N_0(L_\gamma, \Sigma) \leq N_{c'(\gamma)}(\Delta^{\nu_{\gamma}^{*}h}. \Sigma). 
\]

We now estimate the number of eigenvalues of $\Delta^{\nu_{\gamma}^{*}h}$. Let $\{\lambda_i\}_i$ and $\{\mu_i\}_i$ be the eigenvalues of the Laplacians $\Delta^{\nu_{\gamma}^{*}h}$ and $\Delta^{W_\gamma}$. Since $\nu_\gamma$ is a branched covering map from $\Sigma$ to $W_\gamma$, we may apply Tysk's heat trace estimate \cite{T} to $\nu_\gamma$ and obtain for any $t > 0$ that
\begin{equation}
	\sum_{i} e^{-\lambda_i t} \leq \mathrm{deg}(\nu_\gamma)\sum_{i} e^{-\mu_i t}. 
\end{equation}
Combining this with the heat kernel estimate by Cheng and Li \cite{CL}, we have
\begin{align}
	N_{c'(\gamma)}(\Delta^{\nu_{\gamma}^{*}h}, \Sigma)
	&= \sum_{\lambda_i < c'(\gamma)} 1\\
	&\leq e^{c'(\gamma)t}\sum_{\lambda_i < c'(\gamma)} e^{-\lambda_i t}\\
	&\leq e^{c'(\gamma)t} \mathrm{deg}(\nu_\gamma)\sum_{i} e^{-\mu_i t}\\
	&\leq  \mathrm{deg}(\nu_\gamma)C(W_\gamma)\frac{e^{c'(\gamma)t}}{t}, \label{eq:numberbound}
\end{align}
where $C(W_\gamma)$ is the Sobolev constant of $W_\gamma$. Taking a minimum of the right hand side of \eqref{eq:numberbound}, we finally obtain that
\[
	\Ind(\Sigma) \leq c'(\gamma)^{-1}C(W_\gamma)\mathrm{deg}(\nu_\gamma) = C(\gamma) \mathrm{deg}(\nu_\gamma). 
\]
\end{proof}
\section{Index lower bound for anisotropic minimal surfaces}
In this section, we assume that $\Sigma \to \RR^{3}$ is an embedded, complete, non-planner $\gamma$-minimal surface. 

\begin{prop}\label{prop:Courant}
Assume $\Ind(\Sigma)$ is finite. Let $u$ is a nontrivial solution to the equation $L u = 0$ and let $N_u$ be the number of connected components of $\Sigma \setminus u^{-1}(0)$. Then it follows that
\begin{equation}
	\Ind(\Sigma) \geq N_u - 1. 
\end{equation}
\end{prop}
\begin{proof}
Let $I = \Ind(\Sigma)$. Then $u$ is an $(I + 1)$-th eigenfunction of $L$. By the Courant nodal domain theorem, the number of nodal domains of $u$ is not greater than $I + 1$. Therefore, $I \geq N_u -1$. 
\end{proof}

The following lemma asserts that nontrivial solutions of $L u = 0$ can be constructed by the Gauss map $\nu$. 
\begin{lem}\label{lem:gaussJacobi}
Let $\Sigma$ be a $\gamma$-minimal surface in $\RR^{3}$ and let $a \in \RR^{3}$ be a non-zero fixed vector. Then the function $\phi_a = \inn<\nu, a>$ satisfies the equation $L u = 0$.
\end{lem}
\begin{proof}
Under the variation $X_t = X + ta$ of a $\gamma$-minimal immersion $X: \Sigma \to \RR^3$, it is clear that the anisotropic mean curvature $H_\gamma$ does not change. Differentiating $H_\gamma$ in $t$ gives $L \phi_a = 0$. 
\end{proof}

Let us consider the nodal set $\phi_a^{-1}(0)$ of $\phi_a$. If $x \in \phi_a^{-1}(0)$, we have $\inn<\nu(x), a> = \phi_a(x) =  0$. This implies that $\nu(x)$ is contained in the great circle in $\SS^2$ determined as the intersection of $\SS^2$ and the plane which is normal to $a$. Therefore, the nodal set of $\phi_a$ is given by the inverse image by $\nu$ of a great circle. 

We observe the behavior of $\nu$ around critical points. Pick any point $p$ in the critical set $C$ of $\nu$. By Theorem\ref{thm:main1}, we may choose a local coordinate neighborhood $U$ around $p$ which has no other points of $C$. Let $D \subset \SS^2$ be a coordinate disk centered at $\nu(p)$. Up to a linear transformation on $U$, we may assume that $\nu(p) = 0$. Since $\nu$ satisfies the equation $L \nu = 0$, we may also assume that $\Delta \nu + \aniso \nu = 0$ at $p$. 

By the theorem of Aronszajn \cite{A}, the vanishing order of $\nu$ at $p$ is finite. Hence, by the theorem of Bers \cite{B}, there exists a homogeneous polynomial $P$ of order $b(p) + 1$ and a constant $\varepsilon \in (0, 1)$ such that 
\begin{equation}
	\nu(x, y) = P(x, y) + \mathcal{O}(r^{b(p) + 1 + \varepsilon}), 
\end{equation}
where $(x, y)$ denotes the local coordinate around $p$. By Lemma2.4 of \cite{Ch}, there exists a $C^1$-diffeomorphism $\Psi$ around $p$ such that $\nu = P\circ \Psi$. Hence, for any regular value $q'$ of $\nu$ in $D$, we have $\#\nu^{-1}(q') = b(p) + 1$. 

Thus, we may define {\it the branching order of $p \in C$} by the number $b(p)$ defined as above. Considering a triangulation of $\SS^2$ and its pullback on $\Sigma$ via $\nu$, we obtain a Riemann--Hurwitz type formula. 
\begin{prop}\label{prop:RHformula}
For any $\gamma$-anisotropic minimal surface $\Sigma$ which has finite total curvature with $A_\gamma > 0$, it follows that
\begin{equation}
	\chi(\tilde{\Sigma}) = 2\mathrm{deg}(\nu) - \sum_{p \in \tilde{\Sigma}} b(p). 
\end{equation}
\end{prop}

Let $S$ be a great circle in $\SS^2$. The following lemma asserts that the nodal set $\nu^{-1}(S)$ forms a pseudograph on $\Sigma$. 
\begin{lem}\label{lem:pgraph}
$\nu^{-1}(S)$ forms a pseudograph on $\Sigma$. 
\end{lem}
\begin{proof}
Let $q_1, \cdots, q_s$ be singular values of $\nu$ on $S$. We may assume that $q_1, \cdots, q_s$ lie on $S$ in this order. For each $q_i$, we set $\{p^{i}_1, \cdots, p^{i}_{t_i}\} = \nu^{-1}(q_i)$. Set $t = \sum_{i}{t_i}$ and $b = \sum_{i, j}{b^{i}_j}$, where $b^{i}_j = b(p^{i}_j)$. 

Consider a local disk around $p^{i}_j$. By the above argument, the arc $q_iq_{i+1}$ is lifted to $b^{i}_j + 1$ curves starting from $p^{i}_j$ and the terminal points are among $p^{i+1}_1, \cdots, p^{i+1}_{t_{i+1}}$ (here we interpreted as $q_{s+1} = q_1$, etc.). Since $\nu$ is local homeomorphism away from its critical sets, each edge has no self-intersections and any two edges do not intersect at their interiors. Thus, $\nu^{-1}(S)$ forms an embedded pseudograph on $\Sigma$ consisting of $t$-vertices and $b + t$-edges. 
\end{proof}

For a subset $A$ of $\Sigma$, we define the branching order $b(\nu, A)$ of $\nu$ with respect to $A$ by 
\[
	b(\nu, A) = \sum_{p \in A} b(p). 
\]
To show an index lower estimate, we need the following graph theoretic lemma (see \cite{N}). 
\begin{lem}\label{eq:pgraphEuler}
Let $\Gamma$ be an embedded pseudograph on the compact surface $S$ of genus $g$. Suppose $\Gamma$ has $v$-vertices and $e$-edges, and $S \setminus \Gamma$ has $N$-components. Then 
\begin{equation}
	v -e + N \geq 2 - 2g. 
\end{equation}
\end{lem}

\begin{proof}[proof of Theorem \ref{thm:main3}]
Let $\tilde{\Sigma}$ be the compactification of $\Sigma$. By Lemma \ref{lem:pgraph}, $\nu^{-1}(S)$ forms a pseudograph on $\tilde{\Sigma}$ consisting of $t$-vertices and $b+t$-edges where $b$ and $t$ are the numbers defined in the proof of the lemma. Notice that $\nu^{-1}(S)$ coincides with the nodal set of the Jacobi function $\phi_a$ where $a$ is a unit vector in $\RR^3$ which is orthogonal to $S$. Hence, by Lemma \ref{eq:pgraphEuler}, we obtain
\[
	N_{\phi_a} \geq b + 2-2g. 
\] 

Since $b$ is the total branching order $b(\nu, \nu^{-1}(S))$, it follows from Proposition \ref{prop:Courant} that
\[
	\Ind(\Sigma) \geq N_{\phi_a} -1 \geq b(\nu, \nu^{-1}(S)) + 1-2g. 
\]
\end{proof}

\begin{cor}
Let $\Sigma$ be a $\gamma$-anisotropic minimal surface which has finite total curvature with $A_\gamma > 0$. If the image of critical points $\nu(C)$ is contained in some great circle in $\SS^2$, it follows that
\begin{equation}
	\Ind(\Sigma) \geq -\frac{1}{2\pi}\intS K^\Sigma d\Sigma-1. 
\end{equation}
\end{cor}
\begin{proof}
Let $S$ be a great circle which contains the image of critical set $\nu(C)$. By Proposition \ref{prop:RHformula}, we obtain
\begin{equation}\label{eq:totalbranch}
	b(\nu, \nu^{-1}(S)) = \sum_{p \in \tilde{\Sigma}} b(p) = 2\mathrm{deg}(\nu) + 2g - 2. 
\end{equation}
Since the absolute total curvature of $\Sigma$ coincides with $4\pi \mathrm{deg}(\nu)$, it follows from Theorem \ref{thm:main3} that 
\[
	\Ind(\Sigma) \geq 2\mathrm{deg}(\nu) + 2g - 2 +1 - 2g = -\frac{1}{2\pi}\intS K^\Sigma d\Sigma-1. 
\]
\end{proof}

Using Theorem \ref{thm:main3}, we may prove the instability of anisotropic minimal surface with low genus. 
\begin{cor}\label{cor:lowgenus}
Let $\Sigma$ be a $\gamma$-anisotropic minimal surface which has finite total curvature with $A_\gamma > 0$. If the genus $g$ of $\Sigma$ is $0$ or $1$, then $\Sigma$ is unstable. 
\end{cor}
\begin{proof}
If $g=0$, Theorem \ref{thm:main3} gives $\Ind(\Sigma) \geq 1$. 

Assume that $g=1$. Since $\Sigma$ is not a plain, the mapping degree of the Gauss map $\mathrm{deg}(\nu)$ is greater than $0$. It follows from Proposition \ref{prop:RHformula} that 
\begin{equation}
	\sum_{p \in \tilde{\Sigma}} b(p) = 2\mathrm{deg}(\nu) + 2g - 2 \geq 2g = 2. 
\end{equation}
This implies the existence of a great circle $S$ which contains at least two branch points (counted with multiplicity). Hence, we have $b(\nu, \nu^{-1}(S)) \geq 2$. Thus, Theorem \ref{thm:main3} gives $\Ind(\Sigma) \geq 2 + 1 - 2 = 1$. 
\end{proof}

\end{document}